\documentclass[12pt]{amsart}
\usepackage[T1]{fontenc}
\usepackage[osf,sc]{mathpazo}
\usepackage{amsmath,amsthm,amsfonts,latexsym,amscd,amssymb,amsopn,enumerate,hyperref, mathrsfs}
\usepackage[utf8]{inputenc}
\usepackage{amsmath}
\usepackage{tikz-cd}
\usepackage{bm}
\usepackage{graphicx}
\usepackage{geometry}
\hypersetup{
	colorlinks=true,
	linktoc=all,
	linkcolor=blue,
	citecolor=red,
	filecolor=black,
	urlcolor=blue
}

\usepackage{geometry}
\usepackage[inline]{enumitem}
\makeatletter
\newcommand{\inlineitem}[1][]{%
	\ifnum\enit@type=\tw@
	{\descriptionlabel{#1}}
	\hspace{\labelsep}%
	\else
	\ifnum\enit@type=\z@
	\refstepcounter{\@listctr}\fi
	\quad\@itemlabel\hspace{\labelsep}%
	\fi} \makeatother
\newtheorem{thm}{Theorem}[subsection]
\newtheorem{lem}[thm]{Lemma}
\newtheorem{prop}[thm]{Proposition}

\newtheorem{cor}[thm]{Corollary}

\newtheorem{example}[thm]{Example}

\makeatletter
\def\namedlabel#1#2{\begingroup
	\def\@currentlabel{#2}%
	\label{#1}\endgroup
}
\makeatother

\theoremstyle{definition}

\theoremstyle{remark}
\newtheorem{remark}[thm]{Remark}
\numberwithin{equation}{subsection}
\begin{document}
\title{Schur Index and Extensions of Witt-Berman's Theorems}
\author[R. S. Kulkarni]{Ravindra Shripad Kulkarni}
\address{Ravindra Shripad Kulkarni, Bhaskaracharya Pratishthana, 56/14, Off Law College Road, Damle Road, Pune, Maharashtra, 411004, India. \, email: {\tt punekulk@gmail.com}}
   
\author[S. S. Pradhan]{Soham Swadhin Pradhan {*}}
\address{Soham Swadhin Pradhan, School of Mathematics, Harish-Chandra Research Institute, HBNI, Chhatnag Road, Jhunsi, Allahabad, 211 019,  India. \, email: {\tt soham.spradhan@gmail.com}}

\subjclass[2010]{20Cxx}
\keywords{Group Algebras, Primitive Central Idempotents, Schur Index, Induced Representations}
\thanks{* Corresponding author.\\
E-mail addresses: punekulk@gmail.com (R. S. Kulkarni), soham.spradhan@gmail.com (S. S. Pradhan).}
\date{\sc \today}
\begin{abstract}
Let $G$ be a finite group, and $F$ a field of characteristic $0$ or prime to the order of $G$. In $1952$, Witt and in $1956$, Berman independently proved that the number of inequivalent irreducible $F$-representations of $G$ is equal to the number of $F$-conjugacy classes of the elements of $G$, where {\it ``$F$-conjugacy"} was defined in a certain way. 
In this paper, we define $F$-conjugacy on $G$ in a natural way and give a proof of the above {\it Witt-Berman theorem}. In addition, we give an explicit formula for computing a primitive central idempotent {\it (pci)} of the group algebra $F[G]$ corresponding to an irreducible $F$-representation of $G$, which can be obtained from the {\it ``$F$-character table"} of $G$.\\
\noindent 
Let $G$ be a finite group with a normal subgroup $H$ of index $p$, a prime. 
In $1955$, in case $F$ is algebraically closed, Berman computed the primitive central idempotent (pci) of $F[G]$ corresponding to an irreducible  
$F$-representation of $G$, in terms of pci's of $F[H]$. In this paper, we give a complete proof of this {\it Berman's theorem}, and extend this result when $F$ is not necessarily algebraically closed (Theorem \ref{First Theorem}). Also, using classical {\it Schur's theory} and {\it Wedderburn's theory}, we work out decomposition of induced representation of an irreducible $F$-representation of $H$, into irreducible components (Theorem \ref{Second Theorem}). In a separate paper we shall give more applications of theorems
\ref{First Theorem} and \ref{Second Theorem}.

\end{abstract}
\maketitle
\section{Introduction}
Throughout this paper, $G$ denotes a finite group, and $F$ denotes a field of characteristic $0$ or prime to the order of $G$, unless it is explicitly stated. 

Let $H$ be a subgroup of $G$, and $\overline{F}$ the algebraic closure of $F$. Let $\Omega_G$, $\Omega_H$ be the set of all $F$-representations of $G$ and $H$ respectively. In case we know $\Omega_{H}$, then the study of $\Omega_G$ amounts to studying how the induced representations of $G$ induced from representations in $\Omega_H$ decompose into irreducible components. When $F$ is not algebraically closed, there is an additional arithmetic aspect arising from how the cyclotomic polynomials split over $F$. The additional subtlety, in this case, is that of {\it Schur index}. The account of Schur index in the literature, I. Reiner (see \cite{MR0122892}), I. M. Isaacs (see \cite{MR2270898}), gives a view {\it "from the top"}.  For the solution to our problem, we need to develop a view of the Schur index {\it "from the bottom"}.  A combination of these two views leads to a solution to the general problem. 
 
Let $F$ be a field of characteristic $p \geq 0$. An element $g \in G$ is called {\it $p$-regular} if $p = 0$ or if $p > 0$ and the order of g is not divisible by $p$. 
Let $u$ be the least common multiple of the orders of the $p$-regular elements of $G$. Let $\omega$ be a primitive $u$-th root of unity in $\overline{F}$. Let $K = \mbox{Gal}(F(\omega)/F)$, which is an abelian group. Notice that $p \nmid u$ and then we have a homomorphism $\theta$ from the Galois group $K$, into the multiplicative group $\mathbb{Z}^{*}_u$, defined as follows. If $\sigma \in K$, then $\sigma(\omega) = \omega^a$, where $a \in \mathbb{Z}^{*}_u$, and we define $\theta(\sigma) = a$. Let $A = \theta(K) \leq \mathbb{Z}^{*}_u$. Following Berman \cite{Berman}, we say that two $p$-regular elements $x, y \in G$ are conjugate in the sense of Berman, or  ``{\it Berman-conjugate}", if there exists $g \in G$ and $j \in A$ such that $g^{-1}xg = y^j$, that is, $x$ is conjugate to $y^j$. This is an equivalence relation on the set of $p$-regular elements of $G$, and so the set of $p$-regular elements of $G$ is a disjoint
union of Berman conjugacy classes of $p$-regular elements. In $1952$, Witt (see \cite{Witt}) and in $1956$ (see \cite{Berman-1952}, \cite{Berman}), Berman independently proved that {\it the number of equivalence classes of irreducible $F$-representations of $G$ equals the number of Berman conjugacy classes of $p$-regular elements of $G$.} This theorem is known as {\it Witt-Berman theorem}. The result can be found in characteristic $0$ as (see \cite{Curtis-1962}, p. 306, Theorem $42.8$), and the proof is based on {\it Witt-Berman induction theorem} (see \cite{Witt}, \cite{Berman}, \cite{Curtis-1962}, p. $302$, Theorem $42.3$). A simplified proof in characteristic $p$ using Brauer characters can be found in \cite{Reiner-1964}.

In this paper, we prove Witt-Berman theorem when $F$ is a field of characteristic $0$ or prime to the order of $G$. We define $F$-conjugacy as follows. We say two elements $x,y \in G$ are {\it $F$-conjugate}, denoted by $\sim_{F}$, if $\chi(x) = \chi(y)$ for all $F$-characters $\chi$ of $G$. Note that $\sim_{F}$ is an equivalence relation on the elements of $G$. Here, using our definition of $F$-conjugacy, we
prove that $x \sim_{F} y$ if and only if  $x, y$ are Berman conjugate, and so the number of inequivalent irreducible $F$-representations of $G$ is equal to the number of $F$-conjugacy classes of elements of $G$. We also show that the $F$-conjugacy class of an element in $G$ of order $n$, is a union of certain conjugacy classes of $G$, and is determined by the roots of a single irreducible factor of $\Phi_n(X)$, the $n$-th cyclotomic polynomial, over $F$. Witt-Berman theorem allows us to consider {\it $F$-character table}, like usual the character table. We give an explicit formula for {\it primitive central $F$-idempotents} (pci's) of $F[G]$ corresponding to irreducible $F$-representations of $G$, and which can be read from the $F$-character table. The formula is known (see \cite{Lux-2010}, p. 89, Theorem 2.1.6), but we express it in terms of the irreducible $F$-characters and $F$-conjugacy classes. 

Now suppose $G$ contains a normal subgroup of prime index $p$.
In case $F$ is algebraically closed,  it is well known that as a consequence of {\it Clifford's theorem} (see \cite{MR1503352}, \cite{MR2270898}, p. $79$, Theorem $6.2$, \cite{Dornhoff-1971}, p.$72$, Theorem $14.1$), every irreducible representation of $H$ either induces an irreducible representation of $G$, or else it extends to $p$ distinct irreducible representations of $G$.  In $1955$, Berman (see \cite{MR0072139}), actually constructed primitive central idempotents of $F[G]$ in terms of the pcis of $F[H]$. In turn, this helps in actually constructing the matrices of representations of a finite solvable group over an algebraically closed field (see \cite{Isaacs-1993}, p. 21, Section 4). Berman's paper is in Russian. We do not know any reference which gives a complete proof of {\it Berman's theorem}. In this paper, we give such a proof, and extend Berman's theorem to the case when $F$ is not necessarily algebraically closed. This is the first main theorem of the paper. In turn, this helps to construct the pci's of the semisimple group algebra of a finite solvable group. 

Suppose that $G$ contains a normal subgroup $H$ of index $p$, a prime, and $F$ is not necessarily algebraically closed. Then the study of irreducible $F$-representations of $G$ amounts to determining the decomposition of induced representations induced from the irreducible representations of $H$, into irreducible components. In $1996$, Glasby (see \cite{Glasby-1996}) showed that there are six possibilities for the structure of the induced module, when $F$ is not assumed to be algebraically closed and its characteristic is positive. In $2004$, Glasby (see \cite{Glasby-2004}) using techniques from noncommutative ring theory showed that there are even more possibilities, when $F$ is not assumed to be algebraically closed and its characteristic is $0$. 
In this paper, using classical {\it Schur's theory} (see \cite{Schur}, \cite{MR1500832}, \cite{MR0122892}, \cite{MR2270898}, Chapter $10$) and {\it Wedderburn's theory} (see \cite{MR0122892}), we work out decomposition of induced representations into irreducible components. This is the second main theorem of the paper. In addition, we illustrate all the cases of the splitting of idempotents and decomposition of induced representations, by several examples.

This paper is arranged as follows: After setting up the necessary notation, in Sec. \ref{Schur Index}, we give two definitions of Schur index ``from the top" and ``from the bottom". In Sec. \ref{Witt-Berman theorem}, we define $F$-conjugacy, give a complete proof of Witt-Berman theorem (Theorem \ref{$F$-conjugacy}), and an explicit formula for primitive central $F$-idempotents (Theorem \ref{$F$-idempotents}). In Sec. \ref{berman's theorem}, we give a complete proof of Berman's theorem (Theorem \ref{Berman's theorem}). In Sec. \ref{main theorem}, we extend Berman's theorem to the case when $F$ is not necessarily algebraically closed (Theorem \ref{First Theorem}) . In Sec. \ref{Induced Representation}, for a group $G$ with a normal subgroup $H$ of index $p$, a prime, $F$ is not necessarily algebraically closed, we determine the decomposition of induced representation, induced from the irreducible representations of $H$ (Theorem \ref{Second Theorem}), into irreducible components, and finally, in Sec. \ref{Examples}, we illustrate our results, by several examples.

{\bf Notations.} The algebraic closure of $F$ is denoted by $\overline F$. If $V$ is a finite dimensional vector space over $F$, and $\rho:G\rightarrow {\rm GL}(V)$ is a group homomorphism, we say that $\rho$ is an {\it $F$-representation}  of $G$, and denote it by $\rho$. For an irreducible $F$-representation $\rho: G \longrightarrow GL(V)$, its corresponding pci in $F[G]$, we denote by $e_{\rho}$. For a positive integer $n$, a cyclic group of order $n$ is denoted by $C_n$. For a subgroup $H$ of $G$, and an $F$-representation $\eta$ of $H$, the {\it conjugate representation} of $\eta$ by an element $x \in G$ is denoted by $\eta^x$. For a prime $p$, and an indeterminate $X$, we denote $(1 + X + \dots + X^{p-1})/p$ by $e_{X}$. The rest of our notations are standard.
\section{Schur Index, Two Definitions}\label{Schur Index}
Let $\rho:G \rightarrow \textrm{GL}(V)$ be an irreducible $F$-representation of $G.$ We call the $F$-vector space $V$ the {\it representation-space}
of $\rho.$ The multiplication in $G$, extended
by linearity, turns $F[G]$ into an $F$-algebra. The representation $\rho:G \rightarrow \textrm{GL}(V)$ canonically extends to an $F$-algebra homomorphism 
$\rho^{*}: F[G] \rightarrow \textrm{End}(V),$
making $V$ an $F[G]$-module. For notational convenience, we use the same notation $\rho$ instead of $\rho^{*}$. Let $D = \{A \in \textrm{End}(V)| A\, \textrm{commutes with}\, \rho(g)\, \textrm{for all}~g \in G\}.$ By Schur's lemma $D$ is a finite-dimensional $F$-algebra which is a division ring. We call $D,$ the {\it centralizer of $G$ in $\rm{End}(V).$}
We can regard $V$ as a right vector space over $D^o, $ where $D^o$ is $D$ with the opposite multiplication: $x.y$ in $D^o$ equals $yx$ in $D.$ Let $n$ be the dimension of $V$ over $D^o,$ called the {\it reduced dimension} of $V.$ Then $M_n(D^o)$
is abstractly isomorphic to the minimal $2$-sided ideal of $F[G]$ which corresponds to the irreducible representation $\rho$. If we regard $F[G]$ as a $F[G]$-module, then $V$ occurs as a simple submodule of $F[G]$ with multiplicity $n$, or, in different terminology, in the left regular $F$-representation, $\rho$ appears with multplicity $n,$ the reduced dimension of $V.$ 
    
Now we briefly describe classical Schur's theory on group representations (see \cite{Schur},  \cite{MR1500832}) and connection with Wedderburn's theory of semisimple algebras. Let $\rho:G \rightarrow \textrm{GL}(V)$ be an irreducible $F$-representation of $G.$ Let $\overline F$ be the algebraic closure
of $F.$ Then $V \otimes_{F} \overline{F}$ is a  $\overline{F}$-representation of $G$. Schur asserted that $V \otimes_{F} \overline{F}$ decomposes into distinct irreducible $\overline{F}$-representations $\rho_1, \rho_2, \dots, \rho_k$
which occur with the same multiplicity, say, $m$. In other words, if $\rho_i: G \rightarrow \textrm{GL}_{\overline F}W_i,$ then $\rho \otimes_F{\overline F}= m(\rho_1 \oplus \rho_2 \oplus \dots \oplus \rho_k), $ and $V\otimes_F{\overline F} \cong m(W_{1}\oplus W_{2} \oplus \dots \oplus W_{k}).$ Its isotypic components $mW_i$ are 
canonically defined as submodules of $V\otimes_F{\overline F}.$ The number $m$ was later called the {\it Schur index of $\rho$}. 
    
The main new ingredients from the theory of semisimple $F$-algebras are the following. Let $D$ be a finite dimensional $F$-algebra which is a division ring. Let $Z$ be its center. Then $Z$ is a finite dimensional field extension of $F,$ and $D$ has a structure of a $Z$-algebra. Then $[D : Z] = m^2,$ for some natural number $m$. $D$ contains maximal subfields $E$ such that $[E: Z] = m.$ Moreover $D\otimes_ZE$ is isomorphic to $M_m(E)$, the $m x m$ matrix ring with entries in $E$. This is a {\it split} simple $E$-algebra. The {\it Schur index of $D$ (or more generally, $M_n(D),$ for any $n$) over $F$} may be  defined to be $m$. This will be used in the theory of representations of finite groups as follows. 
    
Let $\rho: G \rightarrow \textrm{GL}(V)$ be an irreducible $F$-representation of $G.$ Let the minimal $2$-sided ideal of $F[G]$ corresponding to
$\rho$ be abstractly isomorphic to $M_n(D),$ for suitable $n$, and a division ring $D$. Let $Z$ be the center of $D, [Z:F] = k,$ and $E$ a maximal subfield of $D$.
Then $E$ contains $Z.$ Note that $Z[G] = F[G] \otimes_F Z$ is a sum of certain number of minimal $2$-sided ideals, which correspond 
to irreducible $Z$-representations of $G$. Now $M_n(D)\otimes_{F} Z$ is isomorphic to a summand of $Z[G].$ Since $Z$ is the center of $D,$ we have $M_n(D)\otimes_F Z \cong M_n(D\otimes_F Z) \cong M_n(D + D + \dots +D) \cong M_n(D) + M_n(D) + \dots + M_n(D),$ where in the last two terms there are $k$ summands. Each of these summands is isomorphic to
a minimal $2$-sided ideal of the group algebra $Z[G].$ In terms of representations,   $V\otimes_F Z$ splits into $k$ simple $Z[G]$-summands, $V_1 \oplus V_2 \oplus \dots \oplus V_k.$ Each $V_i$ is a $Z$-vector space, which by restriction of scalars
may be regarded as an $F$-vector space. We have $\textrm{dim}_F V = (\textrm{dim}_F Z)(\textrm{dim}_ZV) = k(\textrm{dim}_ZD) (\textrm{dim}_DV) = km^2n = k\textrm{dim}_ZV_i$ for each $i = 1, 2,\dots, k.$ So, $\rho\otimes_F Z$ splits into $k$ distinct irreducible $Z$-representations of $G.$ 
For more precision, we shall write $M_n(D)_i$ for the $i$-th summand in $M_n(D)\otimes_{F} Z \cong M_n(D) + M_n(D) +  \dots + M_n(D),$ which itself occurs as a summand in $Z[G].$ We have $M_n(D)_i \cong M_n(D), $ for each $i = 1, 2, \dots , k,$ and the corresponding representation-space is $V_i.$ 
    
Let us now consider $E[G] = Z[G] \otimes_Z E.$ Each summand of $E[G]$, isomorphic to 
$M_n(D)\otimes_Z E \cong M_n(D\otimes_Z E) \cong M_n(M_m(E)) \cong M_{nm}(E).$ The last term is a single, split,  $E$-algebra. Let the corresponding irreducible representation-space  be denoted by a $E$-vector space, say $W,$ of  dimension $nm,$ so of dimension  $nm\textrm{dim}_ZE = nm^2$ over $Z.$ More precisely, for each $i = 1, 2,\dots, k,$ corresponding to the summand $M_n(D)_i$ of $E[G]$, we have an irreducible representation-space $W_i.$ 
    
Now consider $V_i\otimes_Z E$. It is  an $E$-representation-space of $G.$ As such it decomposes into irreducible $E$-representation-spaces of $G.$ From the above description, its irreducible components are $E$-representation-spaces isomorphic to $W_i.$ Since $\textrm{dim}_Z V_i\otimes_Z E = nm^2, $ it follows that $V_i\otimes_Z E$ is a sum of $m$ copies of $W_i.$
    
To summarise: {\it Let $\rho:G \rightarrow \textrm{GL}(V)$ be an irreducible $F$-representation of $G.$ Then, with $D, Z, E,$ as defined above,   we have $V\otimes_F E = V\otimes_F Z \otimes_Z E= (V_1 + V_2 + \dots +V_k) \otimes_{Z} E = m(W_1 + W_2 + \dots + W_k), $ a decomposition into distinct representation-spaces $W_1,  W_2,  \dots ,W_k$ of $E$-irreducible representations, each occurring with the same multiplicity $m,$ where $m$ is the Schur index of $\rho$}, and $V_{i} = mW_{i},$ for each $i = 1,2, \dots, k.$
    
Note that the $D$'s occurring in the $F$-representation of finite groups are special.
Namely, $Z =$ the centre of $D$, is a {\it cyclotomic extension} of $F$, that is, it is contained in an extension of the form $F(\zeta_r)$, where $\zeta_r$ is  a primitive $r$-th root of unity. This may be seen as follows.

If $u$ is the exponent of $G$, then by Brauer's theorem $F(\zeta_u)$ is a splitting field for all representations of $G$. So, the center of $M_n(D) \otimes_F F(\zeta_u)$ is a direct sum of certain number of copies of $F(\zeta_u)$, and $Z \otimes_F 1$ is injected in the center of $M_n(D) \otimes_F F(\zeta_u)$.
    
Reiner, (see \cite{MR0122892}), actually gives a different definition of the Schur index. If we consider $F$ as the ``bottom'' and $\overline F$ as the ``top'', Reiner gives a definition of the Schur index from a viewpoint of the ``top''. The fields $Z, E$ mentioned above arise as subfields of $\overline F$, which depend on a specific irreducible $F$-representation. This needs some more terminology, for which we follow Martin Isaacs's textbook, (see \cite{MR2270898}, Chapter 10).
    
Let $\tilde\rho$ be an irreducible representation of $G$ over $\overline{F}.$ Let $\chi = \chi_{\tilde\rho}$ be the character of $\tilde\rho$. Then for each $g$ in $G$, $\chi(g)$ is the sum of eigenvalues of $\tilde\rho(g).$ Let $u$ be the exponent of $G,$ that is the 
l.c.m. of the orders of elements of $G$. Then $\chi(g)$ is a  sum of $u$-th roots of unity. So $\chi(g)$ is an element of the field $F(\zeta_u),$ where $\zeta_u$ is a primitive $u$-th root of unity. Let $F(\chi) = F(\chi_{\tilde\rho})$ be the extension field of $F$ obtained by adjoining all $\chi(g)$'s for $g$ in $G.$ It is a subfield of $F(\zeta_u).$ It is called the {\it character field} of $\tilde{\rho}$ over $F.$ Since  
$F(\zeta_u)$ is an abelian Galois extension of $F,$ it follows that $F(\chi)$ is also an abelian Galois extension field of $F.$ Let $A$ be the Galois group of $F(\chi)$ over $F.$ It is easy to see that for each $\alpha,$ an element of $A,$ the values $\alpha(\chi(g))$ are also values of a character of a representation of $G$ over $\overline{F}.$ In effect, starting with $\tilde \rho,$ we have obtained $|A|$ distinct representations  of $G,$ over $\overline{F}.$ Any two of these representations are called {\it algebraically conjugate}. In this way we have obtained a class of mutually inequivalent algebraically conjugate representations of $G.$ These representations have different characters, but they all have the same character field.
    
Let $Z_1$ denote the field $F(\chi),$ and $e=e_{\tilde\rho}$ be the pci corresponding to $\tilde\rho.$ Then  $e$ is in $Z_1[G],$ and $Z_1[G]e$ is a minimal $2-$sided ideal of $Z_1[G].$ By the Artin-Wedderburn theorem,
$Z_1[G]e \cong M_n(D_1),$ for some $n$ and some division ring $D_1.$ Then $V_1 \cong D_1^n$ is a representation
space of $G$ over $Z_1.$ Then $Z_1$ is  the center of $D_1$ and the dimension of $D_1$ over $Z_1$ is $m^2$ for some
$m.$  If $E_1$ is a maximal subfield of $D_1$, it can serve as a splitting field for $D_1,$ and so 
the representation $\tilde\rho,$ is {\it realisable over $E_1.$} That is, we can choose a basis of the representation space 
$V_1\otimes E_1, $ w.r.t. which all the entries of the matrices $\tilde\rho(g)$ for all $g$ in $G$  lie in $E_1.$ 
Since the dimension of $E_1$ over $Z_1$ is $m,$ and $m$ is the least such dimension, 
we can take the second definition, due to I.~Reiner (see \cite{MR0122892}), of Schur index as the minimum of the dimensions of fields  $\tilde E$ over which the representation $\tilde\rho,$ is {\it realisable over $\tilde E.$}
Let $V$ be the vector space over $F$ obtained from $V_1$ by restriction of scalars from $Z_1$ to $F.$ Consider it 
as a representation space of a representation  $\rho$ of $G.$ It is easy to see that $D_1$ is the centraliser of $G$ in $End V.$
Then the Schur index of $\rho$ according to the first definition, equals the Schur index  of the representation of $\tilde\rho$ 
according to the second definition. These two definitions are related in the following way.
  
{\it Theorem. Let $\rho: G \longrightarrow GL(V)$ be an irreducible $F$-representation of $G$. Let $D, Z, [D : Z ] = m^2, [Z : F] = k$ be as in the above. Then\\
(1) $Z$ is a cyclotomic extension of $F$.\\
(2) $V\otimes_{F}Z$ splits into distinct irreducible $Z$-representations $V_1, V_2, \dots, V_k$. The Galois group Gal$(Z: F)$ acts simply transitively on the set $\{V_1, V_2, \dots, V_k\}$.\\
(3) $V\otimes_{F} {\overline F}$ splits into irreducible $\overline{F}$-representations $W_1, W_2, \dots, W_k$, each with multiplicity $m$.\\
(4) $Z$ is isomorphic to the character field of each of $W_1, W_2, \dots, W_k$.}
\section{$F$-conjugacy, $F$-character table and $F$-idempotents}\label{Witt-Berman theorem}
\subsection{$F$-conjugacy}
It is a well known fact, in the semisimple case, that is, when characteristic of $F$ is $0$ or prime to the order of $G$, character determines representation up to equivalence, that is, if $\chi_{\rho_1} = \chi_{\rho_2}$ then $\rho_1$ is equivalent to $\rho_2$. We say that a character $\chi$ is irreducible if it is not a non-trivial sum of two characters. Therefore, the number of inequivalent irreducible $F$-representations of $G$ is equal to the number of irreducible $F$-characters of $G$.\par
     
Two elements $x, y$ in $G$ are said to be {\it $F$-conjugate}: $x {\sim_F} y$ if for all
finite dimensional $F$-representations $(\rho, V)$ with the characters $\chi_{\rho}$, we have $\chi_{\rho}(x) = \chi_{\rho}(y).$  Notice that $\sim_F$ is an equivalence relation on $G$.
The {\it F-conjugacy} class of an element $x \in G$ consists of all
those elements in $G$, which are $F$-conjugate to $x.$ We denote the $F$-conjugacy class of
$x$ by $C_{F}(x)$ and the conjugacy class of $x$ by $C(x).$
In the following proposition, we describe the factorization of $n$-th cyclotomic
polynomial $\Phi_{n}(X)$ into monic irreducible factors over $F$.
\begin{prop}\label{decomposition cyclotomic}
Let $n$ be a positive integer.
Let $F$ be a field of characteristic $0$ or prime to $n$. Let $\Phi_n(X)=f_1(X)f_2(X)\cdots f_k(X)$ be the decomposition of $\Phi_n(X)$ into irreducible monic polynomials over $F$. Then
\begin{enumerate}
\item The degrees of all $f_i(X)$'s are the same.
\item Let $\zeta$ be a root of one $f_i(X)$. Then all the 
roots of $f_i(X)$ are $\{\zeta^{r_1}, \zeta^{r_2}, \ldots,\zeta^{r_s}\}$, 
where all $r_i$'s are natural numbers with $r_1=1$, and 
the sequence $\{ r_1, r_2,\ldots, r_s \}$ is independent of 
irreducible factors of $\Phi_n(X)$ and any root of $\Phi_n(X)$.
\end{enumerate}
\end{prop}
\begin{proof}
The proof follows from the fact that if $\zeta$ is a primitive $n$-th root of $1 \in F$ then the splitting field of each irreducible factor $f_i(X)$ of $\Phi_n(X)$ over $F$ is $F(\zeta)$, which is also the splitting field of $\Phi_{n}(X)$ over $F$.
\end{proof}
\begin{thm}\label{$F$-conjugacy}
Let $G$ be a finite group and $F$ be a field of characteristic $0$ or prime to the order of $G.$ Then the number of mutually inequivalent irreducible $F$-representations of $G$ is equal to the number of $F$-conjugacy classes of elements of $G$.
\end{thm}
\begin{proof}
Let $u$ be the exponent of $G$. Let $\omega$ be a primitive $u$-th root of $1 \in F$. Let $K$ be the Galois group of $F(\omega)$ over $F$. Let $A$ be the subgroup of $\mathbb{Z}_{u}^*$, which gives all the $F$-automorphisms in $K$. Define the following
relations on $G$:
\begin{align*}
g\sim h & \mbox{ if } g  ~\mbox{is Berman-conjugate to h, that is,}~ g \mbox{ is conjugate to } h^a \mbox{ for some } a\in
A;\\
g\sim_F h & \mbox{ if } \chi(g)=\chi(h) \mbox{ for all 
$F$-characters $\chi$ of $G$}.
\end{align*}
For any $g \in G$, let $[g]_{\sim}$ and $[g]_{\sim_F}$ denotes the equivalence classes of $g$ under $\sim$ and $\sim_F$ respectively.\\
        
Claim: $[g]_{\sim}\subseteq [g]_{\sim_F}$ for all $g\in G$. 
        
Let $h \in [g]_{\sim}$. Then $g$ is conjugate to $h^a$ for some $a \in A$, which implies $\chi(g) = \chi(h^a)$ for all $F$-characters $\chi$. Note that $\sigma : w \longrightarrow w^a$ is an $F$-automorphism of $F(\omega)$. So, for all $F$-characters $\chi$, $\chi(g) = \chi(h^a) = \sigma(\chi(h)) = \chi(h)$, which implies that $h \in [g]_{\sim_F}$. Therefore, $[g]_{\sim}\subseteq [g]_{\sim_F}$ for all $g\in G$.\\
        
Consider the sets 
\begin{align*}
M &=\{ f:G\rightarrow F ~\mid~  \mbox{$f$ is constant on each $\sim$
equivalence class of $G$}\},\\
N &=\{ f:G\rightarrow F ~\mid~  \mbox{$f$ is constant on each $\sim_F$
equivalence class of $G$}\}.
\end{align*}
Then $M$ and $N$ are $F$-vector spaces. Let $\phi_1, \phi_2, \dots, \phi_r$ be the all distinct irreducible $F$-characters of $G$. By definition, 
$\mbox{span}\{ \phi_1,\ldots,\phi_r\}\subseteq N\subseteq M.$
We show that $M=\mbox{span}\{ \phi_1,\ldots,\phi_r\}$,
which will imply that $[g]_{\sim}=[g]_{\sim_F}$, and the proof will be complete.\par
        
Consider $\theta\in M$. As $\theta$ is constant on each conjugacy class of
$G$, we can write $\theta=\sum_{i=1}^k c_i\chi_i$ where $c_i\in F(\omega)$ and
$\chi_i$'s are
irreducible characters of $G$ over $F(\omega)$.
Let $\sigma$ be the $F$-automorphism $\omega\mapsto
\omega^a$ of $F(\omega)$. Then for any $g\in G$, $\theta(g)=\theta(g^a)$, hence
$$\sum_{i=1}^k c_i\sigma(\chi_i(g))=\sum_{i=1}^k
c_i\chi_i(g^a)=\theta(g^a)=\theta(g)=\sum_{i=1}^k c_i\chi_i(g).$$
Therefore $c_i=c_j$ if and only if $\sigma(\chi_i)=\chi_j$ (by linear independence of
$\chi_i$'s over $F(\omega)$). Therefore $\theta$ is an $F(\omega)$-linear
combination of $\phi_i$'s, say 
$$\theta=\sum_{i=1}^r d_i\phi_i.$$\par
        
Claim: Each $d_i$ belongs to $F$.\par
        
Since $\phi_1,\ldots,\phi_r$ are linearly independent over $F$ (see\cite{MR2270898}, Theorem $9.22$), there exists
$x_j\in G$ for $j=1,2,\ldots, r$ such that the $r\times r$ matrix
$(\phi_i(x_j))$ is non-singular. Therefore, $d_i$'s are uniquely determined by
the system of linear equations
$$\theta(x_j)=\sum_{i=1}^r d_i\phi_i(x_j), \hskip5mm j=1,2,\ldots, r.$$
It follows that $d_i=\sigma(d_i)$, for any $F$-automorphism $\sigma$ of
$F(\omega)$. Thus $\theta$ is an $F$-linear combination of $\phi_1,\ldots, \phi_r$. Therefore, span$\{\phi_1,\ldots, \phi_r\} = N = M.$ Therefore, $\{\phi_1, \dots, \phi_r\}$ is a basis of both $N$ and $M$. Consequently, $[g]_{\sim}=[g]_{\sim_F}$, which implies that the number of equivalence classes of elements of $G$ under $\sim$ is equal to the number of $F$-conjugacy classes of elements of $G$. 
Since $r = \mbox{dim} M$ then the number of $F$-irreducible representations of $G$ is equal to the number of $F$-conjugacy classes of elements of $G$. This completes the proof.
\end{proof}
\begin{cor}\rm
The set of inequivalent irreducible $F$-characters of $G$ forms a basis of the space of all functions $f:G \longrightarrow F$, which are constant on each $F$-conjugacy class of $G$.
\end{cor}
\begin{thm}
Let $F$ be a field of characteristic $0$ or prime to the order of $G.$ Let $x$ be an element of order $n$ in $G$.
Then $C_{F}{(x)}$ is equal to $C(x^{r_1})\cup C(x^{r_2})\cup\cdots \cup C(x^{r_s})$, where $r_1 = 1, r_2,\ldots, r_s$ is the sequence associated with $\Phi_n(X)$ as in the Proposition \ref{decomposition cyclotomic}.
\end{thm}
\begin{proof}
Let $u$ be the exponent of $G$. As the order of $x$ is $n$ then $n$ is a divisor of $u$. Let $\omega$ be a primitive $u$-th root and $\zeta$ be a primitive $n$-th root of $1\in F$. Then $\zeta = \omega^a,$ for some positive integer $a$. Let $K$, $L$ be the Galois group of $F(\omega)$ and $F(\zeta)$ over $F$ respectively. Let $A$, $B$ be the subgroup of $\mathbb{Z}^{*}_{u}$, $\mathbb{Z}^{*}_{n}$ respectively, which give all the $F$-automorphisms of $F(\omega)$, $F(\zeta)$ in $K$ and $L$ respectively. Therefore $B = \{\overline{r_1} = \overline{1}, \overline{r_2}, \dots, \overline{r_s}\}$ where $\overline{r_i} = r_{i} (\mbox{mod}n)$.
        
Note that $C_F(x)$ is invariant under all inner automorphisms of $G$. So it is  union of conjugacy classes. We first prove that $C(x)\cup C(x^{r_{2}})\cup\cdots\cup C(x^{r_{s}})
\subseteq C_{F}(x)$. 
        
Let $\rho$ be an arbitrary  $F$-representation of $G$, and $\chi$ be its character.
Then $\chi(x)=\zeta^{u_1}+\zeta^{u_2}+\cdots +\zeta^{u_t}$ for some
$u_i\in\mathbb{Z}$, and it belongs to $F$.
Since $\sigma_i:\zeta\mapsto \zeta^{r_i}$ is an element in $L$, then
$$\chi(x^{r_i})=\zeta^{r_iu_1}+\zeta^{r_iu_2}+\cdots+\zeta^{r_iu_t}
=\sigma_i(\zeta^{u_1}+\cdots+\zeta^{u_t})=\sigma_i(\chi(x))=\chi(x).$$
This implies that $C(x)\cup C(x^{r_{2}})\cup\cdots\cup C(x^{r_{s}})
\subseteq C_{F}(x)$.
        
Now we prove the other inclusion. Let $y$ be an element of $C_{F}(x) = [x]_{\sim_F}$. Since $[x]_{\sim_F} = [x]_{\sim}$ then $y$ is conjugate to $x^a,$ for some $a \in A,$ which implies that $y$ is conjugate to $x^a,$ for some $a \in B.$ Therefore $y \in C(x)\cup C(x^{r_{2}})\cup\cdots\cup C(x^{r_{s}})$, and hence $C_{F}(x) \subseteq C(x)\cup C(x^{r_{2}})\cup\cdots\cup C(x^{r_{s}})$. 
This completes the proof.
\end{proof}
\begin{example}
Let $F = {\mathbb{Q}}$. Then 
$C_{\mathbb{Q}}(x)$ is equal to the union of $C(x^i)$, 
where $i$ runs over the positive integers relatively prime to $n$. 
\end{example}
\begin{example}
Let $F=\mathbb{R}$. Then  $C_{\mathbb{R}}(x)$ 
is equal to the union of $C(x)$ and $C(x^{-1})$.
\end{example}
\begin{example}
Let $F=\mathbb{F}_q$, $q = p^r$, $p$ a prime and $r$ is a positive integer. 
Let $d$ be the smallest positive integer such that $q^d \equiv 1 \pmod n$. Then $C_{F}{(x)}$ is equal to the union of $C(x^{q^i})$ for 
$i= 0,1,\ldots, d-1$.
\end{example}
\begin{cor}
Let $x$ be an element in $G$, and of order $n$. Then
the $F$-conjugacy class of $x$ is uniquely determined by the roots of just one irreducible factor of $\Phi_{n}(X)$ over $F$.
\end{cor}

\subsection{$F$-character table}
By Theorem \ref{$F$-conjugacy}, for a finite group $G$, the number of $F$-conjugacy classes of elements of $G$ is equal to the number of $F$-irreducible representations of $G$. We can list the $F$-character values on $F$-conjugacy classes in the form of a square matrix over $F$, which is called the {\it $F$-character table}. 
The columns of $F$-character table are parametrized by $F$-conjugacy classes and the rows are parametrized by irreducible $F$-characters. Since the number of $F$-conjugacy classes, in general, is less than or equal to the number of conjugacy classes, the size of the matrix representing the $F$-character table is smaller than the usual character table. 
\begin{remark}
Consider the important case $F = \mathbb{Q}$, and $G$ abelian.
Then the character table over $\overline F$ is a $|G| \times |G|$ square matrix. One the other hand, let $|G| = \displaystyle \Pi_i p^i$. Then the character table 
of $G$ over $F$, has size only $\displaystyle \Pi_i p(i)$, where $p(i)$ is the number of partitions of $i$, which depends only on the exponents of primes 
occurring in the prime factorization of $|G|$, and not on the actual primes themselves.
\end{remark}
\subsection{$F$-idempotents}
Now we give a formula for computing the pci's of $F[G]$ in terms of $F$-conjugacy classes. Let $R = F[G]$. Let $e_i$ ($i = 1,2,\ldots, r$) be the pci's of $R$, then each $Re_i$ is a minimal two-sided ideal of $R$. 
Let $V$ denote one of these simple $R$-modules and $e$ be 
the corresponding pci. Then by Schur's lemma, 
End$_{F[G]}(V)$ is a division ring $D$, whose center $Z$ contains $F.$ Let $n= \textrm{dim}_D{V}$. Then $Re$ is abstractly isomorphic 
to $M_n({D}^o)$, $V$ is isomorphic to $(D^o)^n$, and $Re$ is isomorphic to 
the direct sum of $n$ copies of $V$. Let $\dim_{F}{Z} = \delta$ and dim$_{Z}{D} = m^2$. 
Then $\dim_F{V} = nm^{2}\delta$, and so $\dim_{F}{Re} = n^2m^2{\delta}.$
    
Let $L_1, L_2,\dots , L_r$ be the $F$-conjugacy classes of $G$ and $C_{1}, C_{2}, \dots, C_{s}$ be the conjugacy classes of $G$. 
Let $\chi$ be any ${F}$-character of $G$. 
Let $\chi(L_i)$ denote the common value of $\chi$ over $L_i$. 
Let $L_{i}$ be the $F$-conjugacy class of $x$.
We denote the $F$-conjugacy class of $x^{-1}$ by $L_{i}{^{-1}}$. 
For any subset $S$ of $G$, $S^*$ denotes the formal sum of elements of $S.$ We use the above notation in the following theorem.
\begin{thm}\label{$F$-idempotents}
Let $F$ be a field of characteristic $0$ or prime to the order of $G.$
Let $(\rho, V)$ be an irreducible $F$-representation of $G,$ $\chi$ be its character and 
$e$ be the corresponding pci in $F[G]$. Let $n$ be the reduced dimension of $V$. 
Then $$e = \frac{n}{|G|} \sum_{i=1}^r \chi(L_i^{-1} ) L_i^*.$$
\end{thm}
    
\begin{proof}
By the classical theorem of Schur (see Section \ref{Schur Index}), 
$\rho \otimes_{F}{\overline{F}}$ splits into $\delta$ distinct irreducible representations over $\overline{F}$ of degree
$mn$, and with the multiplicity $m$. Let
\begin{align*}
\displaystyle \rho \otimes_{F}{\overline{F}} = m(\bigoplus^{\delta}_{i = 1}\tilde \rho_{i}),
\end{align*}
where $\tilde \rho_{i}$'s are algebraically conjugates over $F$ and $m$ is the Schur index of $\tilde \rho_{i}$ with respect to $F.$
Let $\tilde \chi_{i}$ be the $\overline{F}$-character 
of $\tilde \rho_{i}.$ 
The pci corresponding to the representation $\tilde \rho_{i}$ is given by
$$\tilde e_{i} = \frac{mn}{|G|} \sum_{k = 1}^{s} {\tilde \chi_{i}(C_{{k}}^{-1})}C_{{k}}^*.$$
From the Theorem $3.3.1 (1)$ in \cite{Jespers-2016}, we have $e = \tilde e_{1} + \tilde e_{2} + \cdots +  \tilde e_{\delta}.$
So, $$e = \frac{mn}{|G|} \sum_{k = 1}^{s} \Big{(}\tilde \chi_{1} + \tilde \chi_{2} 
+\cdots +\tilde \chi_{\delta} \Big{)} (C_{{k}}^{-1})C_{{k}}^* = 
\frac{n}{|G|} \sum_{i = 1}^{r}{\chi(L_{{i}}^{-1})}L_{{i}}^*.$$
This completes the proof.
\end{proof}
\begin{cor}
Let $e'' = \sum _{x \in G} {\chi(x^{-1})x}$. 
Then $e'' = \{|G|/n\} e$, and 
${e''}^2 = \{|G|/n\}^2 e^2 = \{|G|/n\}^2 e = \{|G|/n\} e''$. 
So if we know $\chi$, then one can determine $n$. 
Therefore one can read the complete set of pci's of $F[G]$ from the $F$-character table.
\end{cor}
\section{Berman's Theorem}\label{berman's theorem}
Berman (see \cite{MR0072139}), described a method for constructing the irreducible representations of a finite solvable group over an algebraically closed field via the group algebra. For that, he inductively constructed all the pci's of the group algebra of a finite solvable group. In this section, we give a proof of Berman's theorem.
\subsection{A Proof of Berman's Theorem} 
For the proof of Berman's theorem, we need the following lemma.
\begin{lem}\label{lemma4.1}
    Let $G = \langle x | x ^{p} = 1 \rangle$ be the cyclic group of prime order $p$. Let $F$ be an algebraically closed field of characteristic either $0$ or coprime to $p.$ Let $\zeta$ be a primitive $p$-th root of unity. Then the pci's in $F[G]$ are 
    $e_{\zeta^ix},$ where $i$ runs over $0, 1, \dots , p-1.$ 
\end{lem}
The proof of the lemma is a standard calculation and we skip it.

Let $G$ be a finite group and $H$  a normal subgroup of index $p$, a prime. Let 
$G/H = \langle {{\overline x}} \rangle$, and $x$ be a lift of $\bar{x}$ in $G$. Let $F$ be an
algebraically closed field of characteristic either $0$ or coprime to $|G|.$ Let $\overline{C(x)}$ be the conjugacy class sum of $x$ in $F[G]$. Since $\overline{C(x)}$ is a central element in $F[G],$ then $\overline{C(x)}^p$ is a central element of $F[H].$
    
Let $(\eta, W)$ be an irreducible $F$-representation 
of $H$ and $e_{\eta}$ be the pci of $\eta$ in $F[H].$ 
As a consequence of Clifford's theorem (see \cite{MR1503352}), if $\eta \ncong \eta^{x}$, then the induced representation $\eta \uparrow^{G}_{H}$ is an irreducible representation of $G$, and 
$\eta \uparrow^{G}_{H}\cong \eta^{x} \uparrow^{G}_{H} \cong \cdots \cong \eta^{x^{p-1}} \uparrow^{G}_{H} \cong \rho,$ say. It is easy to see that
$e_{\rho} = e_{\eta} + e_{\eta^{x}} + \dots + e_{\eta^{x^{p-1}}}.$
    
Now suppose that $\eta \cong \eta^{x}.$ Then $\overline{C(x)}^pe_{\eta}$ belongs to the center of $F[H]e_{\eta}.$ By Schur's lemma,
$\overline{C(x)}^pe_{\eta} = \lambda e_{\eta},$ where $\lambda \in F.$ We show below that
$x$ in $G - H$ may be chosen so that $\lambda \neq 0.$
    
Let $\mu$ be any $p$-th root of $\lambda,$ in $F.$ Let $c = {\overline{C(x)}}{e_{\eta} }/{\mu}$.
Then $c^p = e_{\eta}.$ Let $\zeta$ 
be a primitive $p$-th root of unity in $F$.  By the Lemma $(\ref{lemma4.1})$, $e_{{\zeta}^ic}e_{\eta}$ are $p$ mutually orthogonal central idempotents in $F[G],$ and $e_{\eta}$ is a sum of $e_{{\zeta}^ic}e_{\eta}$ in $F[G].$ Since $e_{\eta}$ can split into at most $p$ central idempotents, then $\eta\uparrow^G_H$ splits into $p$
distinct irreducible representations of $G$. In fact each of these representations are {\it extensions} of $\eta,$ that is, the $H$-action on the representation space $W$ extends to $p$ distinct   $G$-actions on the same vector space $W.$

It remains to show that $x$ in $G - H$ may be chosen such that $\lambda \neq 0.$ Suppose for all $x$ in $G - H$ we have $\overline{C(x)}^pe_{\eta} = 0.$ Let  $\rho$ be an extension of $\eta$, with corresponding
character $\chi_{\rho}$ and the pci $e_{\rho}$ in $F[G]$.
If $\chi_{\rho}$ vanishes on $G-H$, then
$$e_{\rho} = \frac{\textrm{deg}\eta}{|G|}\sum_{g\in G}
\chi_{\rho}(g^{-1})g = \frac{\textrm{deg}\eta}{p|H|}\sum_{h\in H}
\chi_{\eta}(h^{-1})h = \frac{1}{p}e_{\eta},$$
and this implies that $e_{\rho}$ is {\it not} an idempotent, a contradiction. This completes the proof of Berman's theorem.

We summarise the above discussion in the following theorem:    
\begin{thm}\label{Berman's theorem}
	Let $G$ be a finite group and $H$ be a normal subgroup of index $p$, a prime. 
	Let $G/H = \langle {{xH}} \rangle$, for some $x$ in $G$. 
	Let $F$ be an algebraically closed field of characteristic either $0$ or prime to the order of $G.$ Let $\eta$ be an irreducible representation of $H$ over $F$. 
	We distinguish two cases:
	\begin{itemize}
		\item[(1)] If $\eta \ncong \eta^{x}$, then $\rho \cong \eta\uparrow_H^G$ is irreducible,
		${\rho}\downarrow^{G}_{H} \cong  {\eta} \oplus {\eta^{x}} \oplus \cdots \oplus {\eta^{x^{p-1}}}$, and 
		$$e_{\rho} = e_{\eta} + e_{\eta^{x}} + \cdots + e_{\eta^{x^{p-1}}}.$$
		\item[(2)] If $\eta \cong \eta^{x},$ then $\eta$ extends to $p$ distinct irreducible representations $\rho_{0}, \rho_{1}, \dots ,$ $ \rho_{p-1}$ of $G$ over $F$, and $\eta\uparrow_H^G \cong \rho_{0} \oplus \rho_{1} \oplus \cdots \oplus \rho_{p-1}.$ Correspondingly, $e_{\eta} = e_{\rho_0} + e_{\rho_1} + \cdots + e_{\rho_{p-1}}.$ In the above notation, we may take $e_{\rho_i} = e_{{\zeta}^ic}e_{\eta}, i = 0,1, \dots, p-1.$    
	\end{itemize}    
\end{thm}   
\begin{cor}
	With the hypothesis in the Theorem $B$, for an irreducible character $\chi$ of $G,$ $\chi(x) \neq 0$ for some $x \in G - H$ if and only if $\chi |_{H}$ is an irreducible character of $H.$
\end{cor}
    \section{Extension of Berman's Theorem} \label{main theorem}
Let $H$ be a normal subgroup of prime index $p$ in $G.$  Let $F$ be a field of characteristic $0$ or prime to the order of $G.$ Let $\overline{F}$ be the algebraic closure of $F.$ Let $G/H = \langle{\overline x}\rangle.$ Let $x$ be a lift of $\overline{x}$ in $G$. Let $\eta$ be an irreducible $F$-representation of $H,$ $\psi$ be its corresponding character and $e_{\eta}$ be its corresponding pci in $F[H].$ Let $\overline{C(x)}$ denote the conjugacy class sum of $x$ in $F[G].$ Let $I(\eta) = F[H]e_{\eta}$. Let $Z$ denote the center of $I(\eta) = F[H]e_{\eta}.$  So, $\overline{C(x)}^pe_{\eta} = \lambda e_{\eta},$ where $\lambda \in Z.$\par
    
We have
  $$I(\eta)\otimes_{F}{Z} = F[H]e_{\eta} \otimes_{F}{Z} = Z[H]e_{\eta} \cong M_{n}(D) \oplus M_{n}(D) \oplus \dots \oplus M_{n}(D), (\delta\, \textrm{times}).$$ 
  Note that the above isomorphism is a $Z$-algebra isomorphism. Each summand $M_{n}(D)$ is a central simple algebra over $Z$.  In terms of representations, $\eta \otimes_{F}Z = \eta_1 + \eta_2 + \dots + \eta_{\delta},$ where $\eta_{i}$'s are inequivalent irreducible $Z$-representations of $G$.\par 
  
If $\eta^x \ncong \eta$, then like the algebraically closed case $\eta \uparrow^G_H$ is irreducible, and $\eta\uparrow^G_H \cong \eta^{x}\uparrow^G_H \cong \dots \cong \eta^{x^{p-1}}\uparrow^G_H \cong \rho$, say. Moreover, $e_{\rho} = e_{\eta} + e_{\eta^x} + \dots + e_{\eta^{x^{p-1}}}$.  So, from  now onwards we restrict our attention to the case $\eta^x \cong \eta$.\par
    
 Assume  that $\eta^x \cong \eta.$ Then $e_{\eta}$ is central in $F[G].$ If $e_{\eta}$ remains a pci in $F[G],$ we say $e_{\eta}$ {does \it not split} in $F[G],$ otherwise we say $e_{\eta}$ {\it splits} in $F[G].$\par
    
  Let $\Phi_{p}(X)$ denote the $p$-th cyclotomic polynomial. Let $\Phi_{p}(X) =  f_{1}(X) \dots f_{k}(X)$
  be the factorization into monic irreducible polynomials over $Z$. Then $X^{p} - 1 = f_{0}(X)f_{1}(X) \dots$ $f_{k}(X)$,
 where $f_{0}(X) = X-1$, is the factorization into monic irreducible polynomials over $Z$. By Proposition \ref{decomposition cyclotomic}, for
 $i = 1, 2,\dots, \delta$, the degree of $f_{i}(X)$'s are the same, say, $d$. If $\zeta$ is a root of $f_{i}(X)$, then all the roots of $f_{i}(X)$ are $\zeta, \zeta^{r_{2}}, \dots, \zeta^{r_{d}}$, and also the sequence $\{r_{1} = 1 , r_{2} , \dots, r_{d}\}$ is independent of $f_{i}(X)$ and the roots of $f_{i}(X)$. Now we shall state and prove the first main theorem.
  \subsection{First Main Theorem}
Suppose that $\eta^x \cong \eta$. We consider two cases:
  \begin{center}
  $(1)~\eta_{1}^x \ncong \eta_{1}$ \textrm{and} $(2)~\eta_{1}^x \cong \eta_{1}$.
  \end{center}
In case $(1)$, it is easy to see that $\overline{C(x)}^pe_{\eta} = 0$ for all $x$  in $G  - H$.
    
In case $(2)$, since  $\eta_{1}^x \cong \eta_{1}$,  there exists $y \in G - H$, $\overline{C(y)}^pe_{\eta} = \lambda e_{\eta}$, for some $\lambda \neq 0$. Without loss of generality, we take $y$ to be $x$. So, $\overline{C(x)}^pe_{\eta} = \lambda e_{\eta}$, for some $\lambda \neq 0$.
Suppose $Z$ contains a $p$-th root of $\lambda$, say $\mu$.
Let $c = {\overline{C(x)}e_{\eta}}/\mu$. So $c$ is a central element of $Z[G].$ Then $c^p = e_{\eta}.$ Consider $e_c = ( 1 + c + c^2 + \dots + c^{p-1})/p.$
Then $e_ce_{\eta}$ is a central idempotent in $Z[G].$\par
    
Let $\mu_1, \mu_2, \dots, \mu_p$, be the $p$ distinct $p$-th roots of $\lambda$ in $\overline F.$ Let $c_i = {\overline {C(x)}e_{\eta}}/\mu_i, i = 1, 2, \dots, p$ and 
$e_{c_i} = ( 1 + c_i + c_i^2 + \dots + c_i^{p-1})/p.$ Then $e_{c_i}e_{\eta}, i =1,2, \dots, p$ are mutually orthogonal central idempotents in ${\overline F}[G].$ So, in $\overline F$, $e_{\eta}$ splits into $p$ distinct central idempotents: $e_{c_i}e_{\eta}, i = 1, 2,\dots, p.$ Their sum is $e_{\eta}.$\par
    
As for $Z$ itself, the following three mutually exclusive cases can occur:
\begin{itemize}
\item[(A)] $\lambda$ has  no $p$-th roots in $Z$.
\item[(B)] $\lambda$ has two distinct $p$-th roots in $Z$, then $Z$ contains all $p$, $p$-th roots of $\lambda$, say, $\mu_1, \mu_2,\dots , \mu_p$.
In this case,  $\mu_2/\mu_1 = \zeta$, is a primitive $p$-th root of $1$. So, $\zeta$ lies  in $Z.$ Then $\mu_1\zeta^i, i = 0, 1, 2,\dots, p-1,$ are the distinct $p$-th roots of $\lambda$ in $Z.$
Since $Z$ is a field, $\lambda$ has at most $p$ distinct $p$-th roots in $Z$. So, 
$\{\mu_1, \mu_2,\dots, \mu_p\} = \{\mu_1\zeta^i, i = 0, 1, 2,\dots, p-1\}.$
\item[(C)] $\lambda$ has exactly one $p$-th root of $\lambda$, say $\mu$, in $Z.$
 \end{itemize}
Using all the above notations, we state and prove the first main theorem:
\begin{thm}\label{First Theorem}
Let $G$ be a finite group, and $H$ a normal subgroup of prime index $p$ in
$G$. Let $G/H = \langle{\overline x}\rangle.$ Let $x$ be a lift of $\overline{x}$ in $G.$ Let $F$ be a field of characteristic $0$ or prime to the order of $G$. Let $\eta$ be an irreducible $F$-representation of $H,$ and $e_{\eta}$ be its corresponding pci in $F[H]$. 
Suppose that $\eta^x \cong \eta.$ Let $Z$ denote the center of $F[H]e_{\eta}.$ Then
\begin{itemize}
\item[(1)] If $\eta^{x}_1 \ncong \eta_1$, then $e_{\eta}$ does not split in $F[G].$
\item[(2)] If $\eta^{x}_1 \cong \eta_1$, then it follows that $\overline{C(x)}^{p}e_\eta = \lambda e_{\eta},$ where $\lambda \in Z - \{0\}.$ We consider three subcases:
\item[(A)] If $\lambda$ has no $p$-th root in $Z$, then $e_{\eta}$ does not split in $F[G].$ 
\item[(B)] If $\lambda$ has two distinct $p$-th roots in $Z$, then $e_{\eta}$ splits into $p$ pci's in $F[G].$ They are given by $e_{c_i}e_{\eta},  i =1,2, \dots,p$.
\item[(C)] If $\lambda$ has only one $p$-th root in $Z$, then $e_{\eta}$ splits into $1 + k$ pci's in $F[G].$ Let $\zeta_{ij}, j = 1,2, \dots, d$ be the roots of $f_{i}(X)$ in $\bar{F}$. Let $c_{ij} = \frac{\overline {C(x)}e_{\eta}}{\mu \zeta_{ij}},$
$$ e_{c_{ij}} = (1 + c_{ij} + c_{ij}^2 + \dots + c_{ij}^{p-1})/p, i = 1, 2, \dots, k ; j = 1,2, \dots, d$$
and 
$e_{f_{i}(X)} = \sum^{d}_{j = 1} e_{c_{ij}}e_{\eta}.$
Then $e_{\eta}$ splits into 
$e_{f_{i}(X)}, i =1,2, \dots,k,$ together with
$e_{c}e_{\eta}$.
\end{itemize}
\end{thm}
\begin{proof}
$(1)$ As $\eta^{x}_1 \ncong \eta_1$ then $\eta^{x}_i \ncong \eta_i$, for all $i = 1,2, \dots, \delta$. It is easy to see that $\delta$ is a multiple of $p$. Let $\delta = lp$ where $l$ is a positive integer. Since $\eta^{x}_i \ncong \eta_i$, $\eta_{i} \uparrow^G_H$ is irreducible ${Z}$-representation of $G$, say, $\rho_i$. The pci $e_{\rho_i}$ correspoding to $\rho_i$ is:
$$e_{\rho_i} = e_{\eta_i} + e_{\eta^x_i}+ \dots + e_{\eta^{x^{p-1}}_i}.$$
It is easy to see that 
$e_{\eta} = e_{\rho_1} + e_{\rho_2} + \dots + e_{\rho_l}$
and $e_{\rho_i}$'s are algebraically conjugates over $F$. Therefore, $e_{\eta}$ does not split in $F[G]$. This completes the proof of $(1)$.
        
$(2)$ It is easy to show that if $\eta^{x}_1 \cong \eta_1$, then it follows that  $\overline{C(x)}^{p}e_\eta = \lambda e_{\eta},$ where $\lambda \in Z - \{0\}.$
        
In case $(A)$, $e_{\eta}$ does not split in $Z[G]$, therefore, does not split in $F[G].$
        
In case $(B)$, we see that in $Z[G],$ $e_{\eta}$ spilts into $p$ distinct central idempotents.
They are given by 
$$e_{c_{i}}e_{\eta}, i = 1, 2, \dots, p.$$
But $Z$ is contained in $F[H]$, so each $e_{c_{i}}e_{\eta}$ is defined in $F[G],$ and therefore $e_{\eta}$ splits into $p$ distinct central idempotents in $F[G]$. But $e_{\eta}$ splits into at most $p$ distinct pci's in $F[G].$ Therefore, $e_{\eta}$ splits into $p$ pci's $e_{c_{i}}e_{\eta}, i = 1, 2, \dots, p,$ in $F[G]$.
        
In case $(C)$, $Z$ contains exactly one $p$-th root $\mu$ of $\lambda$.
Then the Newton's identities expressing the power symmetric functions into elementary symmetric functions, it is easy to see that 
$$e_{c_{ij}}e_{\eta}, j = 1, 2, \dots, d$$
add to an expression which has coefficients in $Z.$ Since $Z$ is in $F[H],$ these elements lie in $F[G].$ We denote this sum by $e_{f_{i}(X)}, i =1,2, \dots,k.$ Let 
$$c = \frac {\overline {C(x)}e_{\eta}}{\mu}~\textrm{and}~e_{c} =  (1 + c + c^2 + \dots + c^{p-1})/p.$$
One can easily show that $e_{f_{i}(X)}, i =1,2, \dots,k,$ together with
$e_{c}e_{\eta}$, we get $1 + k$ pci's, in $F[G],$  and their sum is $e_{\eta}.$ This completes the proof of the theorem.
\end{proof}
\section{Decompositon of Induced Representations}\label{Induced Representation}
Let $G$ be a finite group and $H$ be a subgroup in $G$. Let $F$ be a field of characteristic $0$ or prime to the order of $G.$ Let $\eta$ be an irreducible $F$-representation of $H$, $\psi$ be its character and $e_{\eta}$ be its corresponding pci in $F[H]$. Let $V$ be the representation space of $\eta$. Let $I(\eta)$ be the minimal $2$-sided ideal of $F[H]$ corresponding to $\eta$, that is, $I(\eta) = F[H]e_{\eta}$. Let $D$ be the $H$-centraliser of $\eta$. Then $I(\eta) \cong M_n(D)$. So, $\textrm{dim}_F I(\eta) = n^2\textrm{dim}_F D$. Let $Z$ be the center of $I(\eta) = F[H]e_{\eta}$, which is equal to the center of $D$. Let $[Z: F] = \delta$. Let $m$ be the Schur index of $\eta$. So, $\textrm{dim}_ZD = m^2, \textrm{dim} _FD = \delta m^2, \textrm{dim}_F I(\eta) = \delta n^2m^2$.

Let $\phi: F[H]\longrightarrow F[G]$ be the canonical inclusion.
Let $J_1, J_2, \dots ,J_r$ be the minimal $2$-sided ideals of $F[G]$, each of whose intersection with $\phi(I(\eta))$ is non-zero.
So the module induced from $I(\eta)$ is $J_1 + J_2 + \dots + J_r$, as 
the module induced from $F[H]$ is $F[G]\otimes_{F[H]} F[H] = F[G]$.
Note that $I(\eta) = nV$ as $F[H]$-modules. So the induced module of $V$ is contained in the induced module of $I(\eta)$, i.e., $J_1 + J_2 + \dots + J_r$, so realizable in $F[G]$.

Now we assume that $H$ is a normal subgroup of prime index $p$ in $G.$ Let $G/H = \langle{\overline x}\rangle.$ Let $x$ be a lift of $\overline{x}$ in $G$. Let $\eta$ be an irreducible $F$-representation of $H$ such that $\eta^x \cong \eta$. 
Let $\overline{C(x)}$ denote the conjugacy class sum of $x$ in $F[G].$  So, $\overline{C(x)}^pe_{\eta} = \lambda e_{\eta},$ where $\lambda \in Z.$
%
%
Then 
$$I(\eta)\otimes_{F}{Z} = F[H]e_{\eta} \otimes_{F}{Z} = Z[H]e_{\eta} \cong M_{n}(D) \oplus M_{n}(D) \oplus \dots \oplus M_{n}(D), (\delta\, \textrm{times}).$$ 
Each summand $M_{n}(D)$ is a central simple algebra over $Z$. 

In terms of representations, $\eta \otimes_{F}Z = \eta_1 + \eta_2 + \dots + \eta_{\delta},$
where $\eta_{i}$'s are inequivalent irreducible $Z$-representations of $G$. 
In terms of pci's, $e_{\eta} = e_{\eta_1} + e_{\eta_2}+ \dots + e_{\eta_\delta}$ is the splitting of $e_{\eta}$ into pci's in $Z[G]$. So, each $Z[H]e_{\eta_i}$ is abstractly isomorphic to $M_{n}(D)$ as $Z$-algebra.  Let $I(\eta_i) = Z[H]e_{i}$ be the minimal $2$-sided ideal of $Z[H]$ corresponding to $\eta_i$. Let $I(\eta)\uparrow^{G}_{H}$ be  the induced ideal in $F[G]$ from $I(\eta)$ and $I(\eta_i)\uparrow^{G}_{H}$ be the induced ideal in $Z[G]$ from $I(\eta_i)$. Therefore, $I(\eta)\uparrow^{G}_{H} = F[G]{e_{\eta}}$ and $I(\eta_i)\uparrow^{G}_{H} = Z[G]{e_{\eta_i}}$.
\subsection{Second Main Theorem} 
We now state and prove the second main theorem.

\begin{thm} \label{Second Theorem}
Let $G$ be a finite group, and $H$ a normal subgroup of prime index $p$ in
$G$. Let $G/H = \langle{\overline x}\rangle.$ Let $x$ be a lift of $\overline{x}$ in $G.$ Let $F$ be a field of characteristic $0$ or prime to the order of $G$. Let $\eta$ be an irreducible $F$-representation of $H,$ and $e_{\eta}$ be its corresponding pci in $F[H]$. 
Suppose that $\eta^x \cong \eta.$ Let $Z$ denote the center of $I(\eta) = F[H]e_{\eta}.$ Then
\begin{itemize}
\item[(1)] If $\eta^{x}_1 \ncong \eta_1$, then $\eta \uparrow^{G}_{H}$ is either irreducible or equivalent to $p\rho,$ where $\rho$ is the unique extension of $\eta$ to $G.$ 
\item[(2)] If $\eta^{x}_1 \cong \eta_1$, then it follows that $\overline{C(x)}^{p}e_\eta = \lambda e_{\eta},$ where $\lambda \in Z - \{0\}.$ Then we consider three subcases:
\item[(A)] If $\lambda$ has no $p$-th root in $Z$, then $\eta \uparrow^{G}_{H}$ is either irreducible or equivalent to $p\rho,$ $\rho$ is the unique extension of $\eta$ to $G.$ 
\item[(B)] If $\lambda$ has two distinct $p$-th roots in $Z$, then $\eta$ extends to $p$ distinct irreducible $F$-representations of $G$, say, $\rho_{0}, \rho_{1}, \dots, \rho_{p-1}$. Then $\eta \uparrow^{G}_{H} \cong \rho_{0} \oplus \rho_{1} \oplus \dots \oplus \rho_{p-1}.$ 
\item[(C)] If $\lambda$ has only one $p$-th root in $Z$, then  $\eta\uparrow^{G}_{H}$ decomposes into $1 + k$ distinct irreducible $F$-representations of $G$. If $\rho_{0}, \rho_{1}, \dots, \rho_{k}$ be the $1 + k$ distinct irreducible $F$-representations of $G$ appear in $\eta\uparrow^{G}_{H}$ then $\eta \uparrow^{G}_{H} \cong \rho_{0} \oplus s(\rho_{1} \oplus \dots \oplus \rho_{k})$, where $\rho_{0}$ is the unique extension of $\eta$ and $s$ divides  g.c.d.$(m, d)$, where $m$ is the Schur index of $\eta$ and d is the common degree of irreducible factors of $\Phi_{p}(X)$ over $Z.$
\end{itemize}
\end{thm}

\begin{proof}
$(1)$ As $\eta^{x}_1 \ncong \eta_1$ then by Theorem~\ref{First Theorem}, $e_{\eta}$ does not split in $F[G]$. So, $I(\eta)\uparrow^{G}_{H} = F[G]{e_{\eta}}$ is a minimal $2$-sided ideal in $F[G]$. Consequently, $\eta \uparrow^G_H$ is either irreducible or equivalent to $p(\rho)$, where $\rho$ is the unique extension of $\eta$.
This completes the proof of $(1)$.
	
$(2)$ It is easy to show that if $\eta^{x}_1 \cong \eta_1$, then it follows that  $\overline{C(x)}^{p}e_\eta = \lambda e_{\eta},$ where $\lambda \in Z - \{0\}.$
	
In case $(A)$, if $\lambda$ has no $p$-th root in $Z$, then by Theorem~\ref{First Theorem}, $e_{\eta}$ does not split  in $F[G]$. So, $I(\eta)\uparrow^{G}_{H} = F[G]{e_{\eta}}$ is a minimal two-sided ideal in $F[G]$. Consequently, $\eta \uparrow^G_H$ is either irreducible or equivalent to $p(\rho)$, where $\rho$ is the unique extension of $\eta$.
	
In case $(B)$, $\lambda$ has two distinct $p$-th roots in $Z$, then by Theorem~\ref{First Theorem}, $e_{\eta}$ splits into $p$ pci's in F[G]. So $I(\eta)\uparrow^{G}_{H} = F[G]{e_{\eta}}$ is direct sum of $p$ minimal $2$-sided ideals in $F[G]$. Consequently, $\eta \uparrow^G_H$ is direct sum of $p$ distinct extensions of $\eta$, say, $\rho_{0}, \rho_{1}, \dots, \rho_{p-1}$ and $\eta \uparrow^{G}_{H} \cong \rho_{0} \oplus \rho_{1} \oplus \dots \oplus \rho_{p-1}.$

In case $(C)$, $Z$ contains exactly one $p$-th root of $\lambda$. Since $\eta_1^x \cong \eta_1$, $\overline{C(x)}^pe_{\eta_1} = \lambda_1e_{\eta_1}$ for some $\lambda_1 \in Z - \{0\}$ and there exist only one $\mu_1$ in $Z$ such that $\mu_1^p = \lambda_1$. Then
\begin{align*}
I(\eta_{1})\otimes_{Z[H]} Z[G] & \cong M_n(D)\otimes_{Z[H]} Z[G]\\
& = M_n(D)\otimes_{Z}{Z[G/H]}\\
& = M_n(D)\otimes_{Z} {Z[X]/(X^p -1)}\\
& = M_n(D)\otimes \big{(}{Z[X]/(X -1) \oplus Z[X]/(\Phi_p(X)}\big{)} \big{)}\\
& = M_n(D) \oplus M_n(D)\otimes_Z \big{(}Z[X]/(\Phi_p(X)\big{)}\big{)}\\
& = M_n(D) \oplus M_n\big{(}D\otimes_Z Z[X]/(\Phi_p(X)\big{)}\big{)}.
\end{align*}
 
Also, 
\begin{align*}
& M_n(D) \oplus  M_n\big{(}D \otimes _Z Z[X]/(\Phi_p(X)\big{)}\big{)}\\ 
& \cong M_n (D) \oplus M_n \Big{(}D\otimes_Z \frac{Z[X]}{(f_1(X))} \oplus D\otimes_Z \frac{Z[X]}{(f_2(X))} \oplus \dots \oplus D\otimes_Z \frac{Z[X]}{(f_k(X))}\Big{)}\\
& \cong M_n (D) \oplus M_n (M_s(D_1)) \oplus M_n (M_s(D_2)) \oplus \dots \oplus  M_n (M_s(D_k))\\
& \cong M_n(D) \oplus M_{ns}(D_1) \oplus   M_{ns}(D_2) \oplus \dots \oplus M_{ns}(D_k).
\end{align*} 
Note that $\mathbb{Z}_p^*$ permutes the summands in $D\otimes_Z \frac{Z[X]}{(f_1(X))} \oplus D\otimes_Z \frac{Z[X]}{(f_2(X))} \oplus \dots \oplus D\otimes_Z \frac{Z[X]}{(f_k(X))}$, so all summands are abstractly isomorphic, and so all $D_i$'s are abstractly isomorphic,
and the same $s$  works for all $D_i$'s.

So if $V  \cong D^n$ is an $F$-irreducible representation space for $H$, then its induced $G$-representation space
splits into $G$-irreducible representation spaces $\cong V \oplus s (V_1 \oplus V_2 \oplus \dots \oplus V_k)$. If $\rho_{0}, \rho_{1}, \dots, \rho_{k}$ are the $1 + k$ distinct irreducible $F$-representations of $G$ appear in $\eta\uparrow^{G}_{H}$ then $\eta \uparrow^{G}_{H} \cong \rho_{0} \oplus s(\rho_{1} \oplus \dots \oplus \rho_{k})$.
	
We have $[D : Z] = m^2$. $D_1$ is contained in $D \otimes_Z Z(\zeta)$. The latter is a semisimple algebra. $D_1$ is a division ring contained in $D \otimes_Z Z(\zeta)$. Let $Z_1$ be the centre of $D_1$. It contains $Z(\zeta)$. So $[D_1 : Z_1]$ divides $[D\otimes_Z Z(\zeta): Z(\zeta)]$. dim$_Z(D)$ = dim$_{Z(\zeta)}(D \otimes_{Z}(Z(\zeta)))$. Note that $Z_1 = Z(\zeta)$. Therefore, $m^2 = \textrm{dim}_{Z_1}(D \otimes_{Z}(Z_1)) = \textrm{dim}_{Z_1}(M_s(D_1)) = s^2m_1^2$, where $m_1$ is the Schur index of $D_1$. So, we get $m^2 = s^2m_1^2$. Therefore $s^2$ divides $m^2$. This implies that $s$ divides $m$. Moreover, $m_1 = m/s$. It is clear that $s$ divides $d$, and hence $s$ divides g.c.d.$(m, d)$. This completes the proof of the theorem.
\end{proof}

\begin{remark}
Let $G$ be a finite group, and $H$ a normal subgroup of prime index $p$ in
$G$. Let $G/H = \langle{\overline x}\rangle.$ Let $x$ be a lift of $\overline{x}$ in $G.$ Let $F$ be a field of characteristic $0$ or prime to the order of $G$. Let $\eta$ be an irreducible $F$-representation of $H,$ and $e_{\eta}$ be its corresponding pci in $F[H]$. 
If $\eta^x \ncong \eta$ then $\eta \uparrow^G_H$ is irreducible, say, $\rho$ and also if $F[H]e_{\eta} \cong M_{n}(D)$ then $F[G]e_{\rho} \cong M_{np}(D)$.
\end{remark}
\begin{remark}
Now we make some comments on Theorem~\ref{Second Theorem}. We use the notation as in Theorem~\ref{Second Theorem}.
\begin{enumerate}
	\item $\eta^x_1 \ncong \eta_1$.\par
	\noindent
	In this case $\eta \uparrow^G_H$ is either irreducible  or equivalent to $p\rho$, where $\rho$ is the unique extension of $\eta$.\par 
	
	\item[*] In case $\eta \uparrow^G_H$ is irreducible, $\rho$, say, then the Scur index of $\rho$ is $mp$.
	\item[*] In case $\eta \uparrow^G_H$ is equivalent to $p\rho$, where $\rho$ is the unique extension of $\eta$, then the Schur index of $\rho$ is $m$.
	\item $\eta^x_1 \cong \eta_1$. 
	\item[(A)] If $\lambda$ has no $p$-th roots in $Z$, then $\eta \uparrow^G_H$ is either irreducible  or equivalent to $p\rho$, where $\rho$ is the unique extension of $\eta$. 
	\item[*] In case $\eta \uparrow^G_H$ is irreducible, $\rho$, say, then the Scur index of $\rho$ is $m$.
	\item[*] In case $\eta \uparrow^G_H$ is equivalent to $p\rho$, where $\rho$ is the unique extension of $\eta$, then the Schur index of $\rho$ is $m/p$.
	\item[(B)] If $\lambda$ has two distinct $p$-th roots in $Z$, then $\eta$ extends to $p$ distinct irreducible $F$-representations of $G$, say, $\rho_{0}, \rho_{1}, \dots, \rho_{p-1}$. 
	\item[*] In this case, the Schur index of each $\rho_i$ is $m$.
	
	\item[(C)] If $\lambda$ has only one $p$-th root in $Z$, then  $\eta\uparrow^{G}_{H}$ decomposes into $1 + k$ distinct irreducible $F$-representations of $G$. Let $\rho_{0}, \rho_{1}, \dots, \rho_{k}$ be the $1 + k$ distinct irreducible $F$-representations of $G$ appear in $\eta\uparrow^{G}_{H}$. Recall that $\rho_{0}$ is an extension of $\eta$ and $\rho_{1}, \dots, \rho_{k}$ appear in $\eta\uparrow^{G}_{H}$ with multiplicity $s$.
	\item[*] The Schur index of $\rho_0$ is $m$.
	\item[*] The Schur index of each $\rho_i, i =1,2 \dots, k$ is $m/s$, where $s$ is a divisor of g.c.d$(m, d)$.
	\item[*] We always have $dk = p-1$. So deg$~\rho_i = d/s$ deg$~\eta$, for $i = 1, 2, \dots, k$. In the frequently occurring case, $m = 1$, so $s = 1$, and so deg$~\rho_i = d$deg$~\eta$.
\end{enumerate}
\end{remark}
\section{Examples}\label{Examples}
In this section, we give some examples illustrating the cases in Theorem \ref{First Theorem} and in Theorem \ref{Second Theorem}. We shall use the notations as in the previous section.\par
\begin{example}\rm
Consider $G$ to be $Q_{8}$ with a prsentation:
$$G = \langle{x, y, z~|~ x^2 = 1, y^2 = x, z^2 = y^2, z^{-1}yz = xy}\rangle.$$
Take $H = C_{4} = \langle{x, y~|~ x^2 = 1, y^2 = x}\rangle$ and $F = \mathbb{Q}$. Let $\eta$ be the unique faithful irreducible $\mathbb{Q}$-representation of $H$ of degree $2$, and so $\eta^z \cong \eta$. We check that $e_{\eta} = 1 - e_x$, where $e_x = (1 + x)/2$ and $Z \cong \mathbb{Q}(i)$, $i = \sqrt{-1}$. In this case, $e_{\eta}$ does not split, and $\eta \uparrow^G_H$ is irreducible, say, $\rho$. Note that the simple components of $\eta$ and $\rho$ in their Wedderburn decompositions are $\mathbb{Q}(i)$ and $\mathbb{H}_{\mathbb{Q}}$ respectively, where $\mathbb{H}_\mathbb{Q}$, the quaternion algebra over $\mathbb{Q}$. So, the Schur index of $\rho$ is $2$. 
\par

Note that the irreducible $\mathbb{Q}$-representations of $G$ are of degrees $1, 1, 1, 1$ and $4$, whereas the degrees of irreducible ${\overline{\mathbb{Q}}}$-representations of $G$ are $1, 1, 1, 1$  and $2$. Let $V$ be an irreducible representation space of $G$ of dimension $4$ over $\mathbb{Q}$. Here $V \cong \mathbb{H}_{\mathbb{Q}}$ (as a $\mathbb{Q}$-vector space) has reduced dimension $1$, that is, it has dimension $1$ regarded as a right vector space over  $\mathbb{H}_{\mathbb{Q}}$ (as a division ring). As is well known, as rings, $\mathbb{H}_{\mathbb{Q}}\otimes {\overline {\mathbb{Q}}}$ is $M_2({\overline{\mathbb{Q}}})$, the ring of $2 \times 2$ matrices over ${\overline{\mathbb{Q}}}$. Each of the two columns of $M_{2}(\overline{\mathbb{Q}})$ can serve as an irreducible representation space of $G$ over $\overline{\mathbb{Q}}$. Notice that dim$(V \otimes_{{\mathbb{Q}}}{\overline{\mathbb{Q}}}) = 4$ also, and it  is a direct sum of $2$ copies of the same irreducible $2$-dimensional representation space of $G$ over $\overline{\mathbb{Q}}$.  
It was the genius of Schur to realise that  this indeed was a general phenomenon!
\end{example}

\begin{example}\rm    
Consider $G$ to be $C_{7} \rtimes C_3$ with a presentation:
$$G = \langle{x, y~|~ x^7 = y^3 = 1, y^{-1}xy = x^2}\rangle.$$
Take $H = C_{7} = \langle{x~|~ x^7 = 1}\rangle$ and $F = \mathbb{Q}$. Let $\eta$ be the unique faithful irreducible $\mathbb{Q}$-representation of $H$ of degree $6$, and so $\eta^y \cong \eta$. We check that $e_{\eta} = 1 - e_x$, where $e_{x} = (1 + x + \dots + x^6) / 7$, and $Z \cong \mathbb{Q}(\zeta)$, where $\zeta$ is a primitive $7$-th root of unity. In this case, $e_{\eta}$ does not split in $\mathbb{Q}[G]$, and $\eta\uparrow^{G}_{H} \cong 3\rho$, where $\rho$ is an extension of $\eta$. Note that simple components corresponding to $\eta$, $\rho$  are $\mathbb{Q}(\zeta)$ and 
$M_{3}(\mathbb{Q}(\zeta + \zeta ^2 + \zeta^4))$ respectively.
\end{example}
\begin{example}\rm
Let $p$ be a prime. Consider $G$ to be $C_{p^{2}} = \langle{x, y~|~ x^p = 1, y^p = x}\rangle$. Take $H = C_{p} = \langle x~|~ x^p =1 \rangle$, and $F = \mathbb{Q}$.  Let $\eta$ be the unique faithful irreducible $\mathbb{Q}$-representation of $H$ of degree $p-1$. We check that $e_{\eta} = 1 - e_x$, where $e_{x} = (1 + x + \dots + x^{p-1}) / p$, and ${Z} \cong \mathbb{Q}(\zeta_p)$, where $\zeta_p$ is a primitive $p$-th root of unity. In this case, $e_{\eta}$ does not split in $\mathbb{Q}[G]$ and $\eta\uparrow^{G}_{H}$ is irreducible, say, $\rho$. Note that simple components corresponding to $\eta$, $\rho$ are $\mathbb{Q}(\zeta_p)$ and ${\mathbb{Q}(\zeta_{p^2})}$ respectively, where $\zeta_{p^2}$ is a primitive $p^2$-th root of unity.
\end{example}

\begin{example}\rm
Consider $G$ to be $Q_{8} \times_{C_2} C_{4}$ (central product) with a presentation:
$$\langle x,y,z, t~|~ x^{2} = 1, y^{2} = x, z^{2} = y^{2}, z^{-1}yz = xy, t^2 = x,\\ t^{-1}xt = x, t^{-1}yt = y, t^{-1}zt = z \rangle.$$ Take $H = Q_8$ and $F = \mathbb{Q}$. Take $\eta$ to be the unique faithful irreducible $\mathbb{Q}$-representation of $H$ of degree $4$. We check that $e_{\eta} = 1 - e_{x}$, where $e_x = (1 + x)/2$ and $Z \cong \mathbb{Q}$. In this case, $e_{\eta}$ does not split in $\mathbb{Q}[G]$, and $\eta \uparrow ^G_H \cong 2 \rho$, where $\rho$ is an extension of $\eta$. Note that simple components corresponding to $\eta$, $\rho$ are $\mathbb{H}_{\mathbb{Q}}$ and $M_2(\mathbb{Q}(i))$ respectively.
\end{example}
\begin{example}\rm
Let $p$ be a prime. Consider $G$ to be $C_p \times C_p$ with presentation 
$$\langle{x, y~ |~ x^p = y^p = 1, xy = yx}\rangle.$$
Take $H = C_{p} = \langle{x | x^p = 1}\rangle$ and $F = \mathbb{Q}$. Let $\eta$ be the degree $p-1$ unique faithful irreducible $\mathbb{Q}$-representation of $H$. For an indeterminate $X$, let $e_{X} = (1 + X + \dots + X^{p-1})/p$. One can see that $e_{\eta} = 1 - e_{x}$ and ${Z} \cong \mathbb{Q}(\zeta_p)$, where $\zeta_p$ be a primitive $p$-th root of unity. In this case, $e_{\eta}$ splits into $p$ pci in $\mathbb{Q}[G]$. 
Correspondingly, $\eta\uparrow^{G}_{H}$ decomposes into $p$ distinct extensions of $\eta$ to $G$, say, $\rho_0, \rho_{1}, \dots, \rho_{p-1}$. Then $e_{\rho_{i}} = e_{x^i y}(1 - e_x), i = 0,1, \dots, p-1$. Note that the simple component corresponding to each $\rho_i$ in the Wedderburn decomposition is $\mathbb{Q}(\zeta_p)$. 
\end{example}

\begin{example}\rm 
Consider $G$ to be $\mathrm{SL}_2(3)$ with a presentation:
\begin{align*}
\mathrm{SL}_2(3) = \langle x,y,z,t ~|~x^{2} = 1, y^{2} = x, z^{2} = y^{2}, z^{-1}yz = xy, t^{3} = 1, t^{-1}yt = z, t^{-1}zt = yz \rangle.
\end{align*}
Take $H = Q_{8}$ and $F = \mathbb{Q}$.
Let $\eta$ be the degree $4$ unique faithful irreducible $\mathbb{Q}$-representation of $H$. We can see that $e_{\eta} = 1 - e_x$, where $e_x = (1 + x)/2$, and ${Z} \cong \mathbb{Q}$. In this case, $e_{\eta}$ splits into two pci's in $\mathbb{Q}[G]$. Correspondingly, $\eta\uparrow^{G}_{H}$ decomposes into two distinct irreducible $\mathbb{Q}$-representations of $G$, and of degrees $4$, say, $\rho_0, \rho_1$. In fact, $\eta\uparrow^{G}_{H} \cong \rho_{0} \oplus 2 \rho_{1}.$ Note that the simple component corresponding to $\eta$ is $\mathbb{H}_{\mathbb{Q}}$, and the simple components corresponding to $\rho_0$, $\rho_1$ are $\mathbb{H}_{\mathbb{Q}}$, $\mathbb{H}_{\mathbb{Q}(\omega)} \cong M_{2}(\mathbb{Q}(\omega))$ respectively, where $\omega$ is a primitive cube root of unity.
\end{example}
{\bf Acknowledgments:}
The second author would like to thank to Rahul Dattatraya Kitture for useful discussions. He also wishes to express thanks to both the institutes Bhaskaracharya Pratishthana, Pune-India and Harish-Chandra Research Institute (HRI), Prayagraj (Allahabad)-India for giving all the facilities to complete this work.

\end{document}